\documentclass[12pt]{amsart}

\usepackage{amsmath}
\usepackage{amssymb}
\usepackage{indentfirst}
\usepackage{galois}
\usepackage{amsthm}
\usepackage{amsfonts}

\newtheorem{theorem}{Theorem}[subsection]
\newtheorem{proposition}[theorem]{Proposition}
\newtheorem{lemma}[theorem]{Lemma}

\newtheorem*{open1}{Open Question 1}
\newtheorem*{open2}{Open Question 2}

\theoremstyle{definition}

\theoremstyle{remark}
\newtheorem{remark}[theorem]{Remark}

\newtheorem{req}{Requirement}

\numberwithin{equation}{section}

\begin{document}

\bibliographystyle{plain}

\title{Spectra of weighted composition operators induced by automorphisms}

\author[Y.X. Gao and Z.H. Zhou] {Yong-Xin Gao and Ze-Hua Zhou$^*$}
\address{\newline  Yong-Xin Gao\newline School of Mathematics,
Tianjin University, Tianjin 300354, P.R. China.}

\email{tqlgao@163.com}

\address{\newline Ze-Hua Zhou\newline School of Mathematics, Tianjin University, Tianjin 300354,
P.R. China.}
\email{zehuazhoumath@aliyun.com; zhzhou@tju.edu.cn}


\date{}
\thanks{\noindent $^*$Corresponding author.\\
This work was supported in part by the National Natural Science
Foundation of China (Grant Nos. 11371276; 11301373; 11401426).}

\begin{abstract}
In the paper we give the results about the spectra of non-invertible weighted composition operators induced by automorphisms on several Hilbert spaces, such as Hardy-Hilbert space $H^2(\mathbb{D})$ and weighted Bergman spaces $A_\alpha^2(\mathbb{D})$.
\end{abstract}

\maketitle

\section{Introduction}

\subsection{Background}
Let $\mathbb{D}$ be the unit disk on the complex plane $\mathbb{C}$. The Hilbert-Hardy space $H^2(\mathbb{D})$ is the set of analytic functions on $\mathbb{D}$ such that
$$||f||_{H^2}^2=\sup_{0<r<1}\int_{0}^{2\pi}|f(re^{i\theta})|^2\frac{d\theta}{2\pi}<\infty.$$
For $\alpha>-1$, the weighted Bergman space $A_\alpha^2(\mathbb{D})$ is the set of analytic functions on $\mathbb{D}$ such that
$$||f||_{A_\alpha^2}^2=\int_{\mathbb{D}}|f(z)|^2dA_\alpha(z)<\infty,$$
where
$$dA_\alpha(z)=\frac{1+\alpha}{\pi}(1-|z|^2)^\alpha dA(z)$$
and $dA$ is the Lebesgue area measure on $\mathbb{D}$.

Let $\varphi$ be an analytic map from $\mathbb{D}$ into itself and $\psi$ be an analytic function on $\mathbb{D}$. Then we can define a weighted composition operator $C_{\psi,\varphi}$ on each of the spaces above, providing it is bounded, by
$$C_{\psi,\varphi}f=\psi\cdot f\comp\varphi$$
for all $f$ in the space. By taking $\psi\equiv 1$ we get the composition operator $C_\varphi$, and when $\varphi$ is identity on $\mathbb{D}$, we get the analytic Toeplitz  operator $M_\psi$. However, it is not always the case that $C_{\psi,\varphi}$ can be written as the composition of the two operators $M_\psi$ and $C_\varphi$, see \cite{Gun2}.

In 1964, Forelli in his paper \cite{For} showed a strong connection between the weighted composition operators and the isometries on Hardy spaces. Ever since then, the properties of weighted composition operators have been actively
investigated, take \cite{Bou1, Bou2, HO} for examples as some recent results.

The spectra of weighted composition operators were firstly considered in 1978 by Kamowitz, in his papers \cite{Kam1} and \cite{Kam2}. In \cite{Kam1}, he determined the spectra of weighted composition operators with automorphism symbols on the disk algebra $A(\mathbb{D})$.

In 2011, Gunatillake \cite{Gun1} investigated the spectra of $C_{\psi,\varphi}$ on $H^2(\mathbb{D})$ when it is invertible. Then Hyv\"{a}rinen et al. \cite{Hyv} got better results on a large class of analytic function spaces and Chalendar et al. \cite{Par} extended the results to the Dirichlet space.

\subsection{Main Results}

In this paper, we investigate the spectrum of $C_{\psi,\varphi}$ on $H^2(\mathbb{D})$ and $A_\alpha^2(\mathbb{D})$ when $\varphi$ is an automorphism of $\mathbb{D}$. Remember that the automorphisms of $\mathbb{D}$ can be classified according to the location of their fixed points:

$\bullet$ $\varphi$ is called elliptic if it has a unique fixed point in $\mathbb{D}$.

$\bullet$ $\varphi$ is called hyperbolic if it has two distinct fixed points on $\partial\mathbb{D}$.

$\bullet$ $\varphi$ is called parabolic if it has only one fixed point on $\partial\mathbb{D}$.

Gunatillake \cite{Gun1} proved that $C_{\psi,\varphi}$ is invertible on $H^2(\mathbb{D})$ if and only if both $\psi$ and $\frac{1}{\psi}$ belong to $H^\infty(\mathbb{D})$ and $\varphi$ is an automorphism of $\mathbb{D}$.  The same result holds on $A_\alpha^2(\mathbb{D})$. Then he investigated the spectrum of $C_{\psi,\varphi}$ when it is invertible. However, most of his results are got under the assumption that $\psi$ is continuous up to the boundary.

The discussion was finally completed by \cite{Gun1}, \cite{Hyv} and \cite{No}. Their results are as follows.

$\bullet$ Let $\varphi$ be an elliptic automorphism with fixed point $a\in\mathbb{D}$ and $\psi\in H^\infty(\mathbb{D})$. Suppose $C_{\psi,\varphi}$ is invertible. If $\varphi'(a)^n=1$ for some integer $n\in\mathbb{N}$ and $n_0$ is the smallest such integer, then the spectrum of $C_{\psi,\varphi}$ on either $H^2(\mathbb{D})$ or $A_\alpha^2(\mathbb{D})$ is the closure of the set
$$\{\lambda : \lambda^{n_0}=\prod_{k=0}^{n_0-1}\psi(\varphi_k(z)),z\in\mathbb{D}\}.$$
If $\varphi'(a)^n\ne 1$ for all $n\in\mathbb{N}$ and $\psi\in A(D)$, then the spectrum of $C_{\psi,\varphi}$ on either $H^2(\mathbb{D})$ or $A_\alpha^2(\mathbb{D})$ is the circle
$$\left\{\lambda:|\lambda|=|\psi(a)|\right\}.$$

$\bullet$ Let $\varphi$ be a parabolic automorphism with Denjoy-Wolff point $a\in\partial\mathbb{D}$ and $\psi\in A(\mathbb{D})$. If $C_{\psi,\varphi}$ is invertible, then the spectrum of $C_{\psi,\varphi}$ on either $H^2(\mathbb{D})$ or $A_\alpha^2(\mathbb{D})$ is the circle
$$\left\{\lambda:|\lambda|=|\psi(a)|\right\}.$$

$\bullet$ Let $\varphi$ be a hyperbolic automorphism with Denjoy-Wolff point $a\in\partial\mathbb{D}$ and the other fixed point $b\in\partial\mathbb{D}$. Assume that $\psi\in A(\mathbb{D})$. If $C_{\psi,\varphi}$ is invertible, then on $H^2(\mathbb{D})$ the spectrum of $C_{\psi,\varphi}$ is
$$\left\{\lambda : \min\{|\psi(a)|s^{1/2},|\psi(b)|s^{-1/2}\}\leqslant|\lambda|\leqslant \max\{|\psi(a)|s^{1/2},|\psi(b)|s^{-1/2}\}\right\},$$
and on $A_\alpha^2(\mathbb{D})$ the spectrum of $C_{\psi,\varphi}$ is
\begin{align*}
\Big{\{}\lambda : \min\{|\psi(a)|s^{(\alpha+2)/2},&|\psi(b)|s^{-(\alpha+2)/2}\}\leqslant|\lambda|\leqslant
\\&\max\{|\psi(a)|s^{(\alpha+2)/2},|\psi(b)|s^{-(\alpha+2)/2}\}\Big{\}}.
\end{align*}
Here $s=\varphi'(a)^{-1}=\varphi'(b)$.

The results are respectively Theorem 3.1.1 and Theorem 3.2.1 in \cite{Gun1}, Theorem 4.3 and Theorem 4.9 in \cite{Hyv} and Corollary 3.12 in \cite{No}.

There remain the cases when $C_{\psi,\varphi}$ is not invertible on $H^2(\mathbb{D})$ and $A_\alpha^2(\mathbb{D})$. The aim of this paper is to get results for these cases. Moreover, we point out in this paper that for the hyperbolic and parabolic cases, one can replace the rather strong assumption that $\psi$ is continuous up to the boundary with a reasonable weaker one: the continuity of $\psi$ at the fixed points of $\varphi$ is sufficient.

In order to reduce the length of this paper, we will make detailed discussion only on $H^2(\mathbb{D})$. The situation on $A_\alpha^2(\mathbb{D})$ is almost all the same, we will just give a brief illustration in the last section.

Our main results are as follows.

$\bullet$ Let $\varphi$ be an elliptic automorphism and $C_{\psi,\varphi}$ is not invertible. Then the spectrum of $C_{\psi,\varphi}$ on $H^2(\mathbb{D})$ is given by Theorem \ref{main11} and Theorem \ref{main12}. The corresponding result on $A_\alpha^2(\mathbb{D})$ is Theorem \ref{main4}.

$\bullet$ Let $\varphi$ be a hyperbolic automorphism and $\psi\in H^\infty(\mathbb{D})$ is continuous at the fixed points of $\varphi$. If $C_{\psi,\varphi}$ is not invertible, then the spectrum of $C_{\psi,\varphi}$ on $H^2(\mathbb{D})$ is given by Theorem \ref{main2}. The corresponding result on $A_\alpha^2(\mathbb{D})$ is Theorem \ref{main5}.

$\bullet$ Let $\varphi$ be a parabolic automorphism and $\psi\in H^\infty(\mathbb{D})$ is continuous at the fixed point of $\varphi$. If $C_{\psi,\varphi}$ is not invertible, then the spectrum of $C_{\psi,\varphi}$ on $H^2(\mathbb{D})$ is given by Theorem \ref{main3}. The corresponding result on $A_\alpha^2(\mathbb{D})$ is Theorem \ref{main6}.

\subsection{Preliminaries}

The spaces we consider in this paper are Hardy space $H^2(\mathbb{D})$ and weighted Bergman spaces $A_\alpha^2(\mathbb{D})$ for $\alpha>-1$. These spaces are all reproducing kernel Hilbert spaces because the point evaluation functionals are bounded on them. That is, for each $z\in\mathbb{D}$, we can find $K_z$ in the space such that
$$\langle f,K_z\rangle=f(z)$$
for all functions $f$ in the space. On $H^2(\mathbb{D})$ we have
$$K_z(w)=\frac{1}{1-\overline{z}w}$$
and $||K_z||=1/\sqrt{1-|z|^2}$. On $A_\alpha^2(\mathbb{D})$ we have
$$K_z(w)=\frac{1}{(1-\overline{z}w)^{\alpha+2}}$$
and $||K_z||=1/(1-|z|^2)^{(\alpha+2)/2}$.

By $H^\infty(\mathbb{D})$ we mean the space of all bounded analytic functions on $\mathbb{D}$, with the norm $||f||_\infty=\sup_{z\in\mathbb{D}}|f(z)|$. The disk algebra $A(\mathbb{D})$ is the subspace of $H^\infty(\mathbb{D})$ consisting all the analytic function which are continuous on $\overline{\mathbb{D}}$.

For any automorphism $\varphi$ of the unit disk $\mathbb{D}$ and $\psi\in H^\infty(\mathbb{D})$, the operators $C_\varphi$ and $M_\psi$ is always bounded on $H^2(\mathbb{D})$ and $A_\alpha^2(\mathbb{D})$. So the weighted composition operator $C_{\psi,\varphi}$ is also bounded, and we can write $C_{\psi,\varphi}$ as $$C_{\psi,\varphi}=M_\psi C_\varphi.$$
Throughout this paper we will always assume that $\psi$ is not identically zero on $\mathbb{D}$. Otherwise $C_{\psi,\varphi}$ will be zero and things would become trivial.

For any automorphism $\varphi$, we will denote the $n$-th iteration of $\varphi$ by $\varphi_n$ in this paper. For $n=0$ we just set $\varphi_0$ identity on $\mathbb{D}$. Then a simple calculation, see \cite{Gun1} for example, can show that
$$C_{\psi,\varphi}^n=C_{\psi_{(n)},\varphi_n}=T_{\psi_{(n)}}C_{\varphi_n},$$
where $\psi_{(n)}=\prod_{k=0}^{n-1}\psi\comp\varphi_k$.

The next proposition is about the adjoint operator of $C_{\psi,\varphi}$. We list it as a lemma here because it is crucial for our discussion in the paper. The proof is straightforward and we omit it here.

\begin{lemma}\label{base}
Suppose $C_{\psi,\varphi}$ is bounded on $X$, where $X$ is either the Hardy space $H^2(\mathbb{D})$ or a weighted Bergman spaces $A_\alpha^2(\mathbb{D})$. Then we have
\begin{equation}
C_{\psi,\varphi}^*K_z=\overline{\psi(z)}K_{\varphi(z)}
\end{equation}
for all $z\in\mathbb{D}$. Here $C_{\psi,\varphi}^*$ represent the adjoint operator of $C_{\psi,\varphi}$.
\end{lemma}

For an operator $T$ on a Hilbert space, we use $\sigma(T)$ to denote the spectrum of $T$, and $r(T)$ to denote the spectral radius of $T$.

\section{Elliptic Automorphisms}

We will begin with the case when $\varphi$ is an elliptic automorphism of the unit disk $\mathbb{D}$. Suppose the fixed point of $\varphi$ is $a\in\mathbb{D}$ and let $\phi$ be the automorphism that exchange $a$ and $0$. Then $\tilde{\varphi}=\phi^{-1}\comp\varphi\comp\phi$ is an automorphism with fixed point $0$, hence a rotation. The spectrum of $C_{\psi,\varphi}$ depends much on the fact whether this rotation is rational or not. The case when $\tilde{\varphi}$ is a rational rotation will be left to the last part of this section since it is quite simple. Now we will focus our attention on the case when $\varphi$ is conjugated to an irrational rotation.

For any $r\in[0,+\infty)$, we will use $\Gamma_r$ to denote the set $\{w\in\mathbb{C}:|w|=r\}$ throughout this section.

\subsection{Spectral Radius}

Firstly, we should find out the spectral radius of $C_{\psi,\varphi}$ on $H^2(\mathbb{D})$. The following result, known as the Birkhoff's Ergodic Theorem, plays a key role in our discussion here.

\begin{lemma}\label{erg1}
Suppose $(X,\mathfrak{F},\mu)$ is a probability space. Let $T$ be a surjective map from $X$ onto itself such that for any $A\in\mathfrak{F}$, we have $T^{-1}A\in\mathfrak{F}$ and $\mu(T^{-1}A)=\mu(A)$. If $T$ is ergodic, then for any $f\in L^1(X,\mu)$ we have
$$\lim_{n\to\infty}\frac{1}{n}\sum_{k=1}^{n}f(T^kx)=\int_Xfd\mu$$
for almost every $x\in X$.
\end{lemma}
Here by saying a map $T$ is ergodic we mean that $T^{-1}A=A$ implies $\mu(A)$ is $0$ or $1$ for all $A\in\mathfrak{F}$. For a proof of this theorem one can turn to Appendix 3 of \cite{Erg}.

\begin{remark}
The unit circle $\partial\mathbb{D}$, along with the measure $\frac{d\theta}{2\pi}$, is a probability space. If $\varphi$ is a rotation, then $\varphi$ can be seen as a surjective map from $\partial\mathbb{D}$ onto itself. A simple discussion can show that the rotation $\varphi$ is ergodic on $\partial\mathbb{D}$ if and only if it is an irrational one.
\end{remark}

Suppose $\psi\in A(\mathbb{D})\subset H^2(\mathbb{D})$ is not identically zero, then $\log|\psi|$ is in $L^1(\partial\mathbb{D},d\theta)$, see Theorem 2.7.1 in \cite{GTM237}. So by Birkhoff's Ergodic Theorem, if $\varphi$ is an irrational rotation, then
$$\lim_{n\to\infty}\frac{1}{n}\sum_{k=1}^{n}\log |\psi(\varphi_k(z))|=\frac{1}{2\pi}\int_0^{2\pi}\log |\psi(e^{i\theta})|d\theta$$
for almost every $z\in\partial\mathbb{D}$. Or equivalently we have
$$\lim_{n\to\infty}\left(\prod_{k=1}^{n}|\psi(\varphi_k(z))|\right)^{1/n}=\exp\left(\frac{1}{2\pi}\int_0^{2\pi}\log|\psi(e^{i\theta})|d\theta\right)$$
for almost every $z\in\partial\mathbb{D}$.

Moerover, in \cite{Kam1} Kamowitz proved the following result as Lemma 4.4 there.

\begin{lemma}\label{erg2}
Suppose $\varphi$ is an irrational rotation of $\mathbb{D}$ and $\psi\in A(\mathbb{D})$ is not identically zero, then
$$\limsup_{n\to\infty}\left(\sup_{|z|=1}\prod_{k=1}^{n}|\psi(\varphi_k(z))|\right)^{1/n}=\exp\left(\frac{1}{2\pi}\int_0^{2\pi}\log|\psi(e^{i\theta})|d\theta\right).$$
\end{lemma}

When $\psi$ is in $H^2(\mathbb{D})$, it can be written as $\psi=\tau\cdot v$, where $\tau$ is an inner function and $v$ is an outer function. This factorization is unique up to constant factors of modulus one. We call $\tau$ the inner part of $\psi$, and $v$ the outer part of $\psi$.

Recall that a function $\tau\in H^2(\mathbb{D})$ is called an inner function if $|\tau(z)|=1$ for almost every $z\in\partial\mathbb{D}$, and $v\in H^2(\mathbb{D})$ is called an outer function if
$$\frac{1}{2\pi}\int_0^{2\pi}\log|v(e^{i\theta})|d\theta=\log|v(0)|.$$
So we have
\begin{align*}
\frac{1}{2\pi}\int_0^{2\pi}\log|\psi(e^{i\theta})|d\theta=&\frac{1}{2\pi}\int_0^{2\pi}\log|\tau(e^{i\theta})|d\theta
\\&+\frac{1}{2\pi}\int_0^{2\pi}\log|v(e^{i\theta})|d\theta=\log|v(0)|.
\end{align*}

Now we can get our result about the spectral radius of $C_{\psi,\varphi}$ on $H^2(\mathbb{D})$.

\begin{lemma}\label{rad1}
Suppose $\varphi$ is an irrational rotation of $\mathbb{D}$ and $\psi\in A(\mathbb{D})$ is not identically zero. Then the spectral radius of $C_{\psi,\varphi}$ on $H^2(\mathbb{D})$ is no larger than $|v(0)|$, where $v$ is the outer part of $\psi$.
\end{lemma}

\begin{proof}
For $n\in\mathbb{N}$ we have
$$C_{\psi,\varphi}^n=C_{\psi_{(n)},\varphi_n}=T_{\psi_{(n)}}C_{\varphi_n},$$
where $\varphi_n$ is the $n$-th iteration of $\varphi_n$ and $\psi_{(n)}=\prod_{k=0}^{n-1}\psi\comp\varphi_n$. Since $\varphi_n$ is a rotation for each $n\in\mathbb{N}$, by Theorem 3.6 in \cite{CM} we have $||C_{\varphi_n}||=1$. So we only need to calculate the norm of $T_{\psi_{(n)}}$.

By Lemma \ref{erg2} and the maximum modulus principle,
\begin{align*}
\limsup_{n\to\infty}||\psi_{(n)}||_\infty^{1/n}&=\limsup_{n\to\infty}\left(\sup_{|z|=1}\prod_{k=0}^{n-1}|\psi(\varphi_k(z))|\right)^{1/n}
\\&=\exp\left(\frac{1}{2\pi}\int_0^{2\pi}\log|\psi(e^{i\theta})|d\theta\right)=|v(0)|.
\end{align*}
So
\begin{align*}
r(C_{\psi,\varphi})&=\lim_{n\to\infty}||C_{\psi,\varphi}^n||^{1/n}
\\&\leqslant\limsup_{n\to\infty}\left(||T_{\psi_{(n)}}||^{1/n}\cdot||C_{\varphi_n}||^{1/n}\right)
\\&=\limsup_{n\to\infty}||\psi_{(n)}||_\infty^{1/n}=|v(0)|.
\end{align*}
\end{proof}

\begin{remark}
An alternative tool for the discussion above is Weyl's result about the uniform distribution mod $1$, which was use in \cite{Kam1} and \cite{Gun1}. However, Weyl's original result requires that the function $f$ in Lemma \ref{erg1} is Riemann integrable, other than Lebesgue integrable. In fact Weyl's theorem can basically be seen as a special case of the Ergodic Theorem.
\end{remark}

\subsection{The Spectrum}

Now we shall continute our discussion with the help of Birkhoff's Ergodic Theorem. Suppose $\psi\in A(\mathbb{D})$ is not identically zero, then $g_r(e^{i\theta})=\log|\psi(re^{i\theta})|$ belongs to $L^1(\partial\mathbb{D},d\theta)$ for each $r\in(0,1]$. Throughout this section, we will let $\Delta_\psi(r)$ denote the exponent of the integral of $g_r$ on $\partial\mathbb{D}$, i.e.,
$$\Delta_\psi(r)=\exp\left(\frac{1}{2\pi}\int_0^{2\pi}\log|\psi(re^{i\theta})|d\theta\right)$$
for $r\in(0,1]$. For $r=0$, we simply set $\Delta_\psi(0)=|\psi(0)|$.
Then again by Lemma \ref{erg1}, when $\varphi$ is an irrational rotation we have
$$\lim_{n\to\infty}\left(\prod_{k=1}^{n}|\psi(\varphi_k(z))|\right)^{1/n}=\Delta_\psi(r)$$
for almost every $z\in\Gamma_r$.

Note that for any $\psi\in A(\mathbb{D})$, $\Delta_\psi(0)=|\psi(0)|$ is always less than or equal to $\Delta_\psi(1)=|v(0)|$ since $v$ is the outer part of $\psi$. In fact, $\log|\psi|$ is subharmonic on $\mathbb{D}$. Then by Jensen's formula and the Littlewood Subordination Theorem, see Lemma 6.1.15 in \cite{Kra} and Theorem 2.22 in \cite{CM}, one can check easily that $\Delta_\psi(r)$ is increasing on $[0,1]$ and continuous on $[0,1)$.

\begin{remark}\label{r1}
It should be pointed out here that we don't have $\Delta_\psi(r)$ continuous on $[0,1]$ in general. We know that
$$\sup_{0<r<1}\int_0^{2\pi}\log|\psi(re^{i\theta})|d\theta\leqslant\int_0^{2\pi}\log|\psi(e^{i\theta})|d\theta,$$
but the equality does not always hold. However, if $\psi$ has no zero on $\partial\mathbb{D}$, then by Lebesgue's dominated convergence theorem we can know that
$\sup_{0<r<1}\Delta_\psi(r)=\Delta_\psi(1)$, hence $\Delta_\psi(r)$ is continuous on $[0,1]$.
\end{remark}

Now we can start off for finding out the spectrum of $C_{\psi,\varphi}$. We first deal with the case when $\psi$ has zeros on $\partial \mathbb{D}$.

\begin{proposition}\label{sp11}
Suppose $\varphi$ is an irrational rotation of $\mathbb{D}$ and $\psi\in A(\mathbb{D})$ has zeros on $\partial \mathbb{D}$. Then the spectrum of $C_{\psi,\varphi}$ on $H^2(\mathbb{D})$ is the disk
$$\{\lambda : |\lambda|\leqslant|v(0)|\},$$
where $v$ is the outer part of $\psi$.
\end{proposition}

\begin{proof}
By Lemma \ref{rad1}, it is sufficient for us to prove that the set
$$\{\lambda : |\lambda|\leqslant|v(0)|\}$$
is contained in the spectrum of $C_{\psi,\varphi}$.

Since $v$ is an outer function, $v(0)$ can not be zero. Suppose $\varphi(z)=\eta z$ where $|\eta|=1$ and let $\zeta$ be one of the zeros of $\psi$ on $\partial \mathbb{D}$. Then we can find a sequence $\{\zeta_j\}$ on $\partial \mathbb{D}$ such that $\zeta_j$ tends to $\zeta$ as $j\to\infty$ and
\begin{equation}\label{2.1}
\lim_{n\to\infty}\left(\prod_{k=1}^{n}|\psi(\overline{\eta}^k\zeta_j)|\right)^{1/n}=\exp\left(\frac{1}{2\pi}\int_0^{2\pi}\log |\psi(e^{i\theta})|d\theta\right)=|v(0)|.
\end{equation}
for each $j\in \mathbb{N}$.

Now suppose $|\lambda|<|v(0)|$, then we can take $p,q>0$ such that $|\lambda|<p<q<|v(0)|$. By (\ref{2.1}), we can find $n_j\to \infty$ such that
$$\prod_{k=1}^{n_j}|\psi(\overline{\eta}^k\zeta_j)|>q^{n_j}$$
for each $j\in \mathbb{N}$. Then by the continuity of $\psi$ on $\overline{\mathbb{D}}$, we can find $R_j\in (0,1)$ such that
$$\prod_{k=1}^{n_j}|\psi(\overline{\eta}^kr\zeta_j)|>p^{n_j}$$
for all $r\in(R_j,1)$.

Let
$$B_{j,r}(z)=\prod_{k=1}^{n_j}\frac{z-\overline{\eta}^kr\zeta_j}{1-\eta^kr\overline{\zeta_j}z}$$
be the Blaschke product with zeros $\overline{\eta}^kr\zeta_j$, $k=1,2,...,n_j$. Then it is obvious that $\lim_{r\to 1}|B_{j,r}(r\zeta_j)|=1$. So we can find $R_j'\in (0,1)$ such that $|B_{j,r}(r\zeta_j)|>1/2$ for all $r\in(R_j',1)$.

Now take $r_j$ greater than $\max\{R_j,R_j'\}$ such that $\lim_{j\to\infty}r_j=1$ and $\psi$ has no zero on $\Gamma_{r_j}$. Let $z_j=r_j\zeta_j$, then $z_j$ tends to $\zeta$ as $j\to\infty$. So
$$\lim_{j\to\infty}\psi(z_j)=\psi(\zeta)=0.$$

Let
$$h_j=K_{z_j}+\sum_{k=1}^{n_j}\lambda^ku_{j,k}K_{\overline{\eta}^{k}z_j},$$
where
$$u_{j,k}=\left(\prod_{s=1}^{k}\overline{\psi(\overline{\eta}^{s}z_j)}\right)^{-1}.$$
Then by Lemma \ref{base},
\begin{align*}
C_{\psi,\varphi}^*(u_{j,k}K_{\overline{\eta}^{k}z_j})&=\overline{\psi(\overline{\eta}^{k}z_j)}\cdot u_{j,k}K_{\overline{\eta}^{k-1}z_j}
\\&=u_{j,k-1}K_{\overline{\eta}^{k-1}z_j}.
\end{align*}
So we have
$$(C_{\psi,\varphi}^*-\lambda)h_j=\overline{\psi(z_j)}K_{\eta z_j}-\lambda^{n_j+1} u_{j,n_j}K_{\overline{\eta}^{n_j}z_j}.$$
Since
$$|u_{j,n_j}|=\frac{1}{\prod_{k=1}^{{n_j}}|\psi(\overline{\eta}^{k}z_j)|}<\frac{1}{p^{n_j}},$$
we can know that
\begin{align*}
||(C_{\psi,\varphi}^*-\lambda)h_j||&\leqslant||\psi(z_j)K_{\eta z_j}||+||\lambda^{n_j+1} u_{j,n_j}K_{\overline{\eta}^{n_j}z_j}||
\\&\leqslant\left(|\psi(z_j)|+\frac{|\lambda|^{n_j+1}}{p^{n_j}}\right)\frac{1}{\sqrt{1-r_j^2}}.
\end{align*}

However,
\begin{align}\label{last1}
||h_j||&\geqslant\left|\langle h_j,B_{j,r_j}\frac{K_{z_j}}{||K_{z_j}||}\rangle\right|
\\&=\left|\langle K_{z_j},B_{j,r_j}\frac{K_{z_j}}{||K_{z_j}||}\rangle\right|\nonumber
\\&=|B_{j,r_j}(z_j)|\cdot||K_{z_j}||\geqslant\frac{1}{2}\frac{1}{\sqrt{1-r_j^2}}.\nonumber
\end{align}
Therefore
$$\lim_{j\to\infty}\frac{||(C_{\psi,\varphi}^*-\lambda)h_j||}{||h_j||}\leqslant\lim_{j\to\infty}\left(2|\psi(z_j)|+\frac{2|\lambda|^{n_j+1}}{p^{n_j}}\right)=0,$$
which means that $\lambda$ belongs to the approximate spectrum of $C_{\psi,\varphi}^*$. So we have
$$\{\lambda : |\lambda|\leqslant|v(0)|\}\subset\sigma(C_{\psi,\varphi}).$$
\end{proof}

If $\psi$ has no zero on $\partial \mathbb{D}$, then it must vanish somewhere in $\mathbb{D}$. The discussion is divided into two steps, which are Lemma \ref{sp121} and Lemma \ref{sp122} as follows.

\begin{lemma}\label{sp121}
Suppose $\varphi$ is an irrational rotation of $\mathbb{D}$ and $\psi\in A(\mathbb{D})$ has zeros in $\mathbb{D}$. Let $R=\inf\{|z|:\psi(z)=0,z\in\mathbb{D}\}$. If $R>0$, then for each $0<t<\Delta_\psi(R)$, we have $\sigma_{ap}(C_{\psi,\varphi}^*)\cap \Gamma_t\ne\phi$. Here $\sigma_{ap}(C_{\psi,\varphi}^*)$ is the approximate spectrum of $C_{\psi,\varphi}^*$ on $H^2(\mathbb{D})$.
\end{lemma}

\begin{proof}
First note that $R>0$ implies $\Delta_\psi(R)>0$.

Suppose $\varphi(z)=\eta z$ where $|\eta|=1$. Let $z_0$ be one of the zeros of $\psi$ on $\Gamma_R$. Then we can find a sequence $\{z_j\}_{j=1}^{\infty}$ on $\Gamma_R$ such that $z_j$ converges to $z_0$ and
$$\lim_{n\to\infty}\left(\prod_{k=1}^{n}|\psi(\overline{\eta}^{k}z_j)|\right)^{1/n}=\Delta_\psi(R)$$
for each $j\in\mathbb{N}$. Now for any $t_0<\Delta_\psi(R)$, take $p>0$ such that $t_0<p<\Delta_\psi(R)$. Then we can find $n_j\to\infty$ such that
$$\prod_{k=1}^{n_j}|\psi(\overline{\eta}^{k}z_j)|>p^{n_j}$$
for each $j\in\mathbb{N}$.

Let $\lambda_{n,m}=e^{\frac{2m\pi i}{n+1}}t_0$. Then it is easy to check out that
\begin{equation}\label{2.2}
\sum_{m=0}^{n}\lambda_{n,m}^k=0
\end{equation}
for $k=1,2,...,n$.

Now let
$$h_{j,m}=K_{z_j}+\sum_{k=1}^{n_j}\lambda_{n_j,m}^ku_{j,k}K_{\overline{\eta}^{k}z_j},$$
where
$$u_{j,k}=\left(\prod_{s=1}^{k}\overline{\psi(\overline{\eta}^{s}z_j)}\right)^{-1}.$$
Then
$$(C_{\psi,\varphi}^*-\lambda_{n_j,m})h_{j,m}=\overline{\psi(z_j)}K_{\eta z_j}-\lambda_{n_j,m}^{n_j+1} u_{j,n_j}K_{\overline{\eta}^{n_j}z_j}.$$
Since
$$|u_{j,n_j}|=\frac{1}{\prod_{k=1}^{{n_j}}|\psi(\overline{\eta}^{k}z_j)|}<\frac{1}{p^{n_j}},$$
we have
\begin{align*}
||(C_{\psi,\varphi}^*-\lambda_{n_j,m})h_{j,m}||&\leqslant||\psi(z_j)K_{\eta z_j}||+||\lambda_{n_j,m}^{n_j+1} u_{j,n_j}K_{\overline{\eta}^{n_j}z_j}||
\\&\leqslant\left(|\psi(z_j)|+\frac{t_0^{n_j+1}}{p^{n_j}}\right)\frac{1}{\sqrt{1-R^2}}.
\end{align*}
However, (\ref{2.2}) shows that
$$\sum_{m=0}^{n_j}h_{j,m}=(n_j+1)K_{z_j},$$
so for each $j$ we can find $m_j$ such that
$$||h_{j,m_j}||\geqslant||K_{z_j}||=\frac{1}{\sqrt{1-R^2}}.$$

Since $\lambda_{n_j,m_j}\in\Gamma_{t_0}$, by passing to a subsequence, we can assume that $\lambda_{n_j,m_j}$ converges to $\lambda_0\in\Gamma_{t_0}$ as $j\to\infty$. Then we have
\begin{align*}
\frac{||(C_{\psi,\varphi}^*-\lambda_0)h_{j,m_j}||}{||h_{j,m_j}||}&\leqslant\frac{||(C_{\psi,\varphi}^*-\lambda_{n_j,m_j})h_{j,m_j}||}{||h_{j,m_j}||}+|\lambda_{n_j,m_j}-\lambda_0| \\&\leqslant|\psi(z_j)|+\frac{t_0^{n_j+1}}{p^{n_j}}+|\lambda_{n_j,m_j}-\lambda_0|.
\end{align*}
Therefore,
$$\lim_{j\to\infty}\frac{||(C_{\psi,\varphi}^*-\lambda_0)h_{j,m_j}||}{||h_{j,m_j}||}\leqslant|\psi(z_0)|=0.$$
This means that $\lambda_0\in\sigma_{ap}(C_{\psi,\varphi}^*)\cap \Gamma_{t_0}$.
\end{proof}

\begin{lemma}\label{sp122}
Suppose $\varphi$ is an irrational rotation of $\mathbb{D}$ and $\psi\in A(\mathbb{D})$. Then for each $r\in (0,1)$, we have $\sigma_{ap}(C_{\psi,\varphi}^*)\cap \Gamma_{\Delta_\psi(r)}\ne\phi$.
\end{lemma}

\begin{proof}
Suppose $\varphi(z)=\eta z$ where $|\eta|=1$.
Fix $r_0\in (0,1)$, then we can find $z_0\in\Gamma_{r_0}$ such that
\begin{equation}\label{2.3}
\lim_{n\to\infty}\left(\prod_{k=0}^{n-1}|\psi(\eta^kz_0)|\right)^{1/n}=\Delta_\psi(r_0).
\end{equation}
Let
$$P_n=\prod_{k=0}^{n-1}\overline{\psi(\eta^kz_0)}$$
and $\{\lambda_{n,m}\}_{m=1}^n$ be the $n$-th roots of $P_n$. Then by (\ref{2.3}),
$$\lim_{n\to\infty}|\lambda_{n,m}|=\Delta_\psi(r_0)$$
for all $m\in\mathbb{N}$. Moreover we have
\begin{equation}\label{2.6}
\sum_{m=1}^{n}\lambda_{n,m}^{-s}=0
\end{equation}
for $s=1,2,...,n-1$.

Since $\varphi$ is an irrational rotation, we can find $n_j\to\infty$ such that $\eta^{n_j}z_0$ tends to $z_0$. Now let
$$h_{j,m}=K_{z_0}+\sum_{s=1}^{n_j-1}\frac{\prod_{k=0}^{s-1}\overline{\psi(\eta^kz_0)}}{\lambda_{n_j,m}^s}K_{\eta^sz_0}$$
for all $j\in\mathbb{N}$ and $m\leqslant n_j$. Then by (\ref{2.6}) we have $\sum_{m=1}^{n_j}h_{j,m}=n_jK_{z_0}$. So for each $j$ we could find $m_j$ such that
$$||h_{j,m_j}||\geqslant||K_{z_0}||=\frac{1}{\sqrt{1-r_0^2}}.$$
However,
\begin{align*}
(C_{\psi,\varphi}^*-\lambda_{n_j,m_j})h_{j,m_j}&=\lambda_{n_j,m_j}^{1-n_j}P_{n_j}K_{\eta^{n_j}z_0}-\lambda_{n_j,m_j}K_{z_0}
\\&=\lambda_{n_j,m_j}(K_{\eta^{n_j}z_0}-K_{z_0}).
\end{align*}
So
\begin{align*}
\lim_{j\to\infty}\frac{||(C_{\psi,\varphi}^*-\lambda_{n_j,m_j})h_{j,m_j}||}{||h_{j,m_j}||}
&=\sqrt{1-r_0^2}\lim_{j\to\infty}||\lambda_{n_j,m_j}(K_{\eta^{n_j}z_0}-K_{z_0})||
\\&=0.
\end{align*}

Finally, since $\lim_{j\to\infty}|\lambda_{n_j,m_j}|=\Delta_\psi(r_0)$, by passing to a subsequence, we may assume that $\lambda_{n_j,m_j}$ converges to a point $\lambda_0\in\Gamma_{\Delta_\psi(r_0)}$. Then
\begin{align*}
\lim_{j\to\infty}\frac{||(C_{\psi,\varphi}^*-\lambda_0)h_{j,m_j}||}{||h_{j,m_j}||}
&\leqslant\lim_{j\to\infty}\frac{||(C_{\psi,\varphi}^*-\lambda_{n_j,m_j})h_{j,m_j}||}{||h_{j,m_j}||}+\lim_{j\to\infty}|\lambda_{n_j,m_j}-\lambda_0|
\\&=0.
\end{align*}
Thus, $\lambda_0$ belongs to the approximate spectrum of $C_{\psi,\varphi}^*$.
\end{proof}

Combining Lemma \ref{sp121} and Lemma \ref{sp122} with some more discussions, we could get the result as follows.

\begin{proposition}\label{sp12}
Suppose $\varphi$ is an irrational rotation of $\mathbb{D}$ and $\psi\in A(\mathbb{D})$. If $\psi$ has no zero on $\partial\mathbb{D}$ and $\psi(z_0)=0$ for some $z_0\in\mathbb{D}$. Then the spectrum of $C_{\psi,\varphi}$ on $H^2(\mathbb{D})$ is the disk
$$\{\lambda : |\lambda|\leqslant|v(0)|\},$$
where $v$ is the outer part of $\psi$.
\end{proposition}

\begin{proof}
Let $R=|z_0|$, then $R\in[0,1)$. If $R$ is greater than zero, without loss of generality we can assume that $\psi$ has no zero in $\{z\in\mathbb{D}:|z|<R\}$.

Since $\psi$ has no zero on $\partial\mathbb{D}$, Remark \ref{r1} shows that
$$|v(0)|=\Delta_\psi(1)=\sup_{0<r<1}\Delta_\psi(r).$$
So
\begin{align*}
\{\lambda : |\lambda|\leqslant|v(0)|\}&=\{\lambda:|\lambda|\leqslant\Delta_\psi(R)\}\cup\left(\bigcup_{R<r\leqslant 1}\Gamma_{\Delta_\psi(r)}\right)
\\&=\overline{A\cup B},
\end{align*}
where
$A=\bigcup_{0<t<\Delta_\psi(R)}\Gamma_t$
and
$B=\bigcup_{R<r<1}\Gamma_{\Delta_\psi(r)}$. If $\Delta_\psi(R)=0$, we just set $A=\phi$. Note that this can only happen when $R=0$. By Lemma \ref{sp121} we have
\begin{equation}\label{2.4}
\sigma_{ap}(C_{\psi,\varphi}^*)\cap \Gamma_t\ne\phi
\end{equation}
for all $0<t<\Delta_\psi(R)$. And by Lemma \ref{sp122} we have
\begin{equation}\label{2.5}
\sigma_{ap}(C_{\psi,\varphi}^*)\cap \Gamma_{\Delta_\psi(r)}\ne\phi
\end{equation}
for all $r\in(R,1)$.
Now we will show that both $A$ and $B$ are contained in the spectrum of $C_{\psi,\varphi}$.

For any fixed $\lambda_0\in\sigma_{ap}(C_{\psi,\varphi}^*)\cap(A\cup B)$, we can find a sequence $f_n$ in $H^2(\mathbb{D})$ such that $||f_n||=1$ and $||(C_{\psi,\varphi}^*-\lambda_0)f_n||$ converges to zero. Since
$$C_{\psi,\varphi}^*(1)=C_{\psi,\varphi}^*K_0=\overline{\psi(0)},$$
we have that $||(C_{\psi,\varphi}^*-\lambda_0)(1)||=|\lambda_0-\overline{\psi(0)|}$. Suppose $\lambda_0\in A$, then $|\lambda_0|<\Delta_\psi(R)$. Since $\psi$ has no zero in $\{z\in\mathbb{D}:|z|<R\}$, $\log|\psi|$ is harmonic in $\{z\in\mathbb{D}:|z|< R\}$. So
$$\int_{0}^{2\pi}\log|\psi(Re^{i\theta})|d\theta=\sup_{0<r<R}\int_{0}^{2\pi}\log|\psi(re^{i\theta})|d\theta=\log|\psi(0)|,$$
which means that $\Delta_\psi(R)=\Delta_\psi(0)=|\psi(0)|$. Therefore,
\begin{align*}
||(C_{\psi,\varphi}^*-\lambda_0)(1)||&\geqslant|\psi(0)|-|\lambda_0|
\\&=\Delta_\psi(R)-|\lambda_0|
\\&>0.
\end{align*}
On the other hand, if $\lambda_0\in B$, then $|\lambda_0|=\Delta_\psi(r_0)$  for some $r_0\in(R,1)$, then
\begin{align*}
||(C_{\psi,\varphi}^*-\lambda_0)(1)||&\geqslant|\lambda_0|-|\psi(0)|
\\&=\Delta_\psi(r_0)-|\psi(0)|
\\&>0.
\end{align*}
The last equality follows from Jensen's formula and the fact that $\psi$ has zeros in $\{z:|z|<r_0\}$. In a word, we always have $||(C_{\psi,\varphi}^*-\lambda_0)(1)||$ greater than zero.

Now we can assert that there exist $\delta\in (0,1)$ such that $|\langle f_n,1\rangle|<\delta$ for all $n\in\mathbb{N}$. Otherwise, we can find $|\xi|=1$ and a subsequence of $\{f_n\}$, say $\{f_{n_j}\}$, such that $\langle f_{n_j},\xi\rangle$ tends to $1$. This means that $||f_{n_j}-\xi||$ converges to zero, hence $||(C_{\psi,\varphi}^*-\lambda_0)f_{n_j}||$ converges to $||(C_{\psi,\varphi}^*-\lambda_0)(\xi)||>0$, which is a contradiction.

Suppose $\varphi(z)=\eta z$ where $|\eta|=1$. Let $g_n=T^*_zf_n=T_{\overline{z}}f_n$, then
\begin{equation}\label{last2}
||g_n||^2=||f_n||^2-|\langle f_n,1\rangle|^2>1-\delta^2.
\end{equation}
However, since $C_{\psi,\varphi}T_z=\eta T_zC_{\psi,\varphi}$, we have that
\begin{align*}
||(C_{\psi,\varphi}^*-\eta\lambda_0)g_n||&=||\eta T^*_zC_{\psi,\varphi}^*f_n-\eta\lambda_0T^*_zf_n||
\\&\leqslant||T^*_z||\cdot||(C_{\psi,\varphi}^*-\lambda_0)f_n||
\\&=||(C_{\psi,\varphi}^*-\lambda_0)f_n||.
\end{align*}
So we have $||(C_{\psi,\varphi}^*-\eta\lambda_0)g_n||$ converges to zero, which means that $\eta\lambda_0$ belongs to the approximate spectrum of $C_{\psi,\varphi}^*$. Hence the approximate spectrum of $C_{\psi,\varphi}^*$ contains the points $\eta^k\lambda_0$ for all $k\in\mathbb{N}$. As a result, (\ref{2.4}) implies that $\Gamma_t$ is contained in $\sigma_{ap}(C_{\psi,\varphi}^*)$ for all $0<t<\Delta_\psi(R)$ and (\ref{2.5}) implies that $\Gamma_{\Delta_\psi(r)}$ is contained in $\sigma_{ap}(C_{\psi,\varphi}^*)$ for all for all $r\in(R,1)$, since $\varphi$ is an irrational rotation. So we have
$$\{\lambda : |\lambda|\leqslant|v(0)|\}=\overline{A\cup B}\subset\sigma(C_{\psi,\varphi}^*).$$

Finally, Lemma \ref{rad1} shows that
$$\sigma(C_{\psi,\varphi}^*)\subset\{\lambda : |\lambda|\leqslant|v(0)|\}.$$
Therefore
$$\sigma(C_{\psi,\varphi})=\sigma(C_{\psi,\varphi}^*)=\{\lambda : |\lambda|\leqslant|v(0)|\}.$$
\end{proof}

Now we can get our final result here as follows.

\begin{theorem}\label{main11}
Suppose $\varphi$ is an elliptic automorphism of $\mathbb{D}$ with fixed point $a\in\mathbb{D}$ and $\psi\in A(\mathbb{D})$. Assume that $C_{\psi,\varphi}$ is not invertible. If $\varphi'(a)^k\ne 1$ for all $k\in\mathbb{N}$, then the spectrum of $C_{\psi,\varphi}$ on $H^2(\mathbb{D})$ is the disk
$$\{\lambda : |\lambda|\leqslant|v(a)|\},$$
where $v$ is the outer part of $\psi$.
\end{theorem}

\begin{proof}
If $a=0$, then $\varphi$ is an irrational rotation. The result follows directly from Proposition \ref{sp11} and Proposition \ref{sp12}.

For the general case, let $\phi=\phi_a$ be the automorphism that exchanges $a$ and $0$. Then $\tilde{\varphi}(z)=\phi^{-1}\comp\varphi\comp\phi(z)=\varphi'(a)z$, hence $\tilde{\varphi}$ is an irrational rotation. A simple calculation shows that
\begin{align*}
C_\phi C_{\psi,\varphi}C_\phi^{-1}&=C_\phi C_{\psi,\varphi}C_{\phi^{-1}}
\\&=C_{\tilde{\psi},\tilde{\varphi}},
\end{align*}
where $\tilde{\psi}=\psi\comp\phi$. So $C_{\psi,\varphi}$ is similar to $C_{\tilde{\psi},\tilde{\varphi}}$, which means that they have the same spectrum. Since $v$ is the outer part of $\psi$, it is not hard to see that $v\comp\phi$ is the outer part of $\tilde{\psi}$. So again by Proposition \ref{sp11} and Proposition \ref{sp12}, we have
\begin{align*}
\sigma(C_{\psi,\varphi})&=\sigma(C_{\tilde{\psi},\tilde{\varphi}})
\\&=\{\lambda : |\lambda|\leqslant|v\comp\phi(0)|\}
\\&=\{\lambda : |\lambda|\leqslant|v(a)|\}.
\end{align*}
\end{proof}

\subsection{Rational Rotation}

When $\varphi$ is conjugated to a rational rotation, things become much easier. In fact, it has been solved by Gunatillake in \cite{Gun1}, see his remark after Theorem 3.1.1 there. However, his proof just shows that for each interior point $\lambda$ in $\sigma(C_{\psi,\varphi})$, $C_{\psi,\varphi}-\lambda$ is not surjective. We will deal this case in a new way here, and then we can get a stronger result that the range of $C_{\psi,\varphi}-\lambda$ is not even dense.

\begin{theorem}\label{main12}
Suppose $\varphi$ is an elliptic automorphism of $\mathbb{D}$ with fixed point $a\in\mathbb{D}$ and $\psi\in H^\infty(\mathbb{D})$. Assume that $C_{\psi,\varphi}$ is not invertible. If $\varphi'(a)^n=1$ for some integer $n\in\mathbb{N}$ and $n_0$ is the smallest such integer, then the spectrum of $C_{\psi,\varphi}$ on $H^2(\mathbb{D})$ is the closure of the set
$$\Lambda=\{\lambda : \lambda^{n_0}=\prod_{k=0}^{n_0-1}\psi(\varphi_k(z)),z\in\mathbb{D}\}.$$
Moreover, each point in $\Lambda$ belongs to the compression spectrum of $C_{\psi,\varphi}$.
\end{theorem}

\begin{proof}
Since $\varphi'(a)^{n_0}=1$, we have $\varphi_{n_0}$ identity on $\mathbb{D}$. So
$$C_{\psi,\varphi}^{n_0}=C_{\psi_{(n_0)},\varphi_{n_0}}=T_{\psi_{(n_0)}},$$
where $\psi_{(n_0)}=\prod_{k=0}^{n_0-1}\psi\comp\varphi_k$. So the spectrum of $C_{\psi,\varphi}^{n_0}$ is exactly the spectrum of $T_{\psi_{(n_0)}}$, which is the closure of the set
$$\{\lambda : \lambda=\psi_{(n_0)}(z),z\in\mathbb{D}\}.$$
Therefore, by the spectral mapping theorem, the spectrum of $C_{\psi,\varphi}$ is contained in the closure of the set
$$\{\lambda : \lambda^{n_0}=\psi_{(n_0)}(z),z\in\mathbb{D}\}=\Lambda.$$
So what remains is to show that each point in $\Lambda$ is contained in the compression spectrum of $C_{\psi,\varphi}$.

For this end, fix any $\lambda_0\in\Lambda$, then there exist $z_0\in\mathbb{D}$ such that
$$\lambda_0^{n_0}=\prod_{k=0}^{n_0-1}\psi(\varphi_k(z_0)).$$
If $\lambda_0=0$, then $\psi(\varphi_{k_0}(z_0))$ for some $k_0\geqslant 0$. So
$$C_{\psi,\varphi}^*K_{\varphi_{k_0}(z_0)}=\overline{\psi(\varphi_{k_0}(z_0))}K_{\varphi_{k_0+1}(z_0)}=0.$$
Thus $0$ belongs to the point spectrum of $C_{\psi,\varphi}^*$. Otherwise, let
$$h=K_{z_0}+\sum_{k=1}^{n_0-1}\frac{\prod_{s=0}^{k-1}\overline{\psi(\varphi_s(z_0))}}{\overline{\lambda}_0^k}K_{\varphi_k(z_0)}.$$
Then the minimality of $n_0$ guarantee that $h\ne0$. It is easy to check that $C_{\psi,\varphi}^*h=\overline{\lambda}_0h$. Thus $\overline{\lambda}_0$ belongs to the point spectrum of $C_{\psi,\varphi}^*$. This means that $\lambda_0$ belongs to the compression spectrum of $C_{\psi,\varphi}$.
\end{proof}

\section{Hyperbolic \& Parabolic Automorphisms}

\subsection{Spectral Radius}

If $\varphi$ is a hyperbolic or parabolic automorphism of $\mathbb{D}$, then for any compact set $K$ in $\mathbb{D}$, $\varphi_n$ converges to one of its fixed points uniformly on $K$. We call it the Denjoy-Wolff point of $\varphi$.

Both hyperbolic and parabolic automorphisms have simple models for iteration on half-planes. For example, let $H^+=\{w\in\mathbb{C}:Im w>0\}$ be the upper half-plane on $\mathbb{C}$ and $\varphi$ a hyperbolic automorphism of $\mathbb{D}$ with Denjoy-Wolff point $a$ and the other fixed point $b$. Suppose $\sigma$ is a biholomorphic map from $\mathbb{D}$ onto $H^+$, then $\sigma$ can extend to the boundary. If we have $\sigma(a)=\infty$ and $\sigma(b)=0$, then $\Phi=\sigma\comp\varphi\comp\sigma^{-1}$ is an automorphism of $H^+$ that fixes $0$ and $\infty$, hence a dilation. More accurately, we have
$$\Phi(w)=\sigma\comp\varphi\comp\sigma^{-1}(w)=\varphi'(a)w$$
for all $w\in H^+$. Or equivalently,
$$\sigma\comp\varphi(z)=\varphi'(a)\cdot\sigma(z)$$
for all $z\in\mathbb{D}$.

If $\varphi$ a parabolic automorphism of $\mathbb{D}$ with Denjoy-Wolff point $a$, we can take $\sigma$ such that $\sigma(a)=\infty$. Then $\Phi=\sigma\comp\varphi\comp\sigma^{-1}$ is an automorphism of $H^+$ that fixes $\infty$ only, hence a translation. By choosing $\sigma$ properly, we can make $\Phi(w)=w+1$ or $\Phi(w)=w-1$ for all $w\in H^+$. Then we have
$$\sigma\comp\varphi(z)=\sigma(z)\pm 1$$
for all $z\in\mathbb{D}$.

These models facilitate us in calculating the spectral radius of $C_{\psi,\varphi}$. In fact, what we will do to prove the following two lemmas are basically the same as what have been done in \cite{Gun1} and \cite{Hyv}. However, we point out here that the assumptions on $\psi$ in \cite{Gun1} and \cite{Hyv} are too strong: there is no need to require that $\psi$ belongs to the disk algebra $A(D)$, and the continuity of $\psi$ at the fixed points of $\varphi$ is sufficient for us to get the following lemmas. So we just give brief proofs here.

\begin{lemma}\label{rad2}
Suppose $\varphi$ is a parabolic automorphism of $\mathbb{D}$ with Denjoy-Wolff point $a\in\partial\mathbb{D}$ and $\psi\in H^\infty(\mathbb{D})$ is continuous at the point $a$. Then the spectral radius of $C_{\psi,\varphi}$ on $H^2(\mathbb{D})$ is no larger than $|\psi(a)|$.
\end{lemma}

\begin{proof}
Suppose $\Phi$ is a translation on $H^+$. Without loss of generality, we may assume that $\Phi(w)=w+1$. Let $r\in(0,+\infty)$, then for any $w\in H^+$, at most $2[r]+1$ elements of the sequence $\{\Phi_k(w)\}_{k=0}^{\infty}$ lies in the set $\{w\in H^+ : |w|\leqslant r\}$. By our discussion above, it is equivalent to say that for any neighbourhoods $V$ of $a$ , we can find $N\in\mathbb{N}$ such that at most $N$ elements of the sequence $\{\varphi_k(z)\}_{k=0}^{\infty}$ lies in the set $\mathbb{D}\backslash V$ for all $z\in\mathbb{D}$. By choosing $V$ small enough, we can assume $|\psi(z)|<|\psi(a)|+\epsilon$ when $z\in V_a$ for any given $\epsilon>0$. Let $\psi_{(n)}=\prod_{k=0}^{n-1}\psi\comp\varphi_k$, then for $n>N$ we have
\begin{align*}
||\psi_{(n)}||_{\infty}&=\sup_{z\in\mathbb{D}}\prod_{k=0}^{n-1}|\psi\comp\varphi_k(z)|
\\&\leqslant ||\psi||_{\infty}^N\cdot\left(|\psi(a)|+\epsilon\right)^{n-N}.
\end{align*}
So
\begin{align*}
\limsup_{n\to\infty}||\psi_{(n)}||_{\infty}^{1/n}&\leqslant\lim_{n\to\infty}\left(||\psi||_{\infty}^{N/n}\cdot\left(|\psi(a)|+\epsilon\right)^{1-N/n}\right).
\\&=|\psi(a)|+\epsilon.
\end{align*}
Since $\epsilon$ is arbitrary, we have
$$\limsup_{n\to\infty}||\psi_{(n)}||_{\infty}^{1/n}\leqslant|\psi(a)|.$$

Now for $n\in\mathbb{N}$,
$$C_{\psi,\varphi}^n=C_{\psi_{(n)},\varphi_n}=T_{\psi_{(n)}}C_{\varphi}^n.$$
So
\begin{align*}
r(C_{\psi,\varphi})&=\lim_{n\to\infty}||C_{\psi,\varphi}^n||^{1/n}
\\&\leqslant\limsup_{n\to\infty}\left(||T_{\psi_{(n)}}||^{1/n}\cdot||C_{\varphi}^n||^{1/n}\right)
\\&=r(C_\varphi)\cdot\limsup_{n\to\infty}||\psi_{(n)}||_\infty^{1/n}
\\&\leqslant|\psi(a)|\cdot r(C_\varphi).
\end{align*}
Here $r(C_\varphi)$ is the spectral radius of composition operator $C_\varphi$. By Theorem 3.9 in \cite{CM}, we know that $r(C_\varphi)=\varphi'(a)^{-1/2}=1$. So we have
$$r(C_{\psi,\varphi})\leqslant|\psi(a)|.$$
\end{proof}

\begin{lemma}\label{rad3}
Suppose $\varphi$ is a hyperbolic automorphism of $\mathbb{D}$ with Denjoy-Wolff point $a\in\partial\mathbb{D}$ and the other fixed point $b$. Assume that $\psi\in H^\infty(\mathbb{D})$ is continuous at the points $a$ and $b$. Then the spectral radius of $C_{\psi,\varphi}$ on $H^2(\mathbb{D})$ is no larger than $\max\{|\psi(a)|\varphi'(a)^{-1/2},|\psi(b)|\varphi'(b)^{-1/2}\}$.
\end{lemma}

\begin{proof}
Let $V_a$ and $V_b$ be neighbourhoods of $a$ and $b$ respectively, then by a similar discussion as in Lemma \ref{rad2}, we can find $N\in\mathbb{N}$ such that at most $N$ elements of the sequence $\{\varphi_k(z)\}_{k=0}^{\infty}$ lies in the set $\mathbb{D}\backslash(V_a\cup V_b)$ for all $z\in\mathbb{D}$. By choosing $V_a$ and $V_b$ small enough, we can assume
$$|\psi(z)|\varphi'(z)^{-1/2}<|\psi(a)|\varphi'(a)^{-1/2}+\epsilon$$
when $z\in V_a$ and
$$|\psi(z)|\varphi'(z)^{-1/2}<|\psi(b)|\varphi'(b)^{-1/2}+\epsilon$$
when $z\in V_b$, for any given $\epsilon>0$. Let $\psi_{(n)}=\prod_{k=0}^{n-1}\psi\comp\varphi_k$, then for $n>N$ we have
\begin{align*}
\Big|\Big|\frac{\psi_{(n)}}{(\varphi_n')^{1/2}}\Big|\Big|_{\infty}&=\sup_{z\in\mathbb{D}}\prod_{k=0}^{n-1}\left|\frac{\psi}{(\varphi')^{1/2}}\comp\varphi_k(z)\right|
\\&\leqslant P^N\cdot\left(Q+\epsilon\right)^{n-N},
\end{align*}
where $P=||\psi\cdot(\varphi')^{-1/2}||_\infty$ and $Q=\max\{|\psi(a)|\varphi'(a)^{-1/2},|\psi(b)|\varphi'(b)^{-1/2}\}$.
Since $\epsilon$ is arbitrary, we have
$$\limsup_{n\to\infty}\Big|\Big|\frac{\psi_{(n)}}{(\varphi_n')^{1/2}}\Big|\Big|_{\infty}^{1/n}\leqslant Q.$$
The proof of Theorem 4.6 in \cite{Hyv} shows that $$||C_{(\varphi_n')^{1/2},\varphi_n}||=1$$
for all $n\in\mathbb{N}$. So we have
\begin{align*}
r(C_{\psi,\varphi})&=\lim_{n\to\infty}||C_{\psi_{(n)},\varphi_n}||^{1/n}
\\&\leqslant\limsup_{n\to\infty}\left(\Big|\Big|\frac{\psi_{(n)}}{(\varphi_n')^{1/2}}\Big|\Big|_{\infty}^{1/n}\cdot||C_{(\varphi_n')^{1/2},\varphi_n}||^{1/n}\right)
\\&\leqslant Q.
\end{align*}
\end{proof}

\begin{remark}\label{r2}
We'd like to point out here that the proofs above actually show the following fact: let $\varphi$ be a hyperbolic or parabolic automorphism and $K$ a subset of $\mathbb{D}$, then $\varphi_n$ converges to the Denjoy-Wolff point of $\varphi$ uniformly on $K$ if and only if the Denjoy-Wolff point of $\varphi^{-1}$ lies outside the closure of $K$.
\end{remark}

\subsection{The Spectra}

If $\varphi$ is a hyperbolic or parabolic automorphism of $\mathbb{D}$, then for any $z\in\mathbb{D}$, $\{\varphi_k(z)\}_{k=0}^{\infty}$ is an interpolating sequence for $H^\infty(\mathbb{D})$. By Carleson' condition, the product $\delta_z=\prod_{k=1}^{\infty}d(z,\varphi_k(z))$ is positive for each $z\in\mathbb{D}$. Here
$$d(z,w)=\left|\frac{z-w}{1-\overline{w}z}\right|$$
is the pseudo-hyperbolic distance between $z$ and $w$. The next two lemmas show that $\delta_z$ is in fact bounded away from zero with respect to $z$.

\begin{lemma}\label{int1}
Suppose $\varphi$ is a hyperbolic automorphism of $\mathbb{D}$. Then there exist $\delta>0$ such that
$\prod_{k=1}^{\infty}d(z,\varphi_k(z))>\delta$ for all $z\in\mathbb{D}$.
\end{lemma}

\begin{proof}
Since $\varphi$ is an hyperbolic automorphism, we can find a biholomorphic map $\sigma$ from $\mathbb{D}$ to the upper half plane $H^+$ such that
$$\sigma\comp\varphi(z)=s\sigma(z)$$
for all $z\in\mathbb{D}$. Here $s=\varphi'(a)\in(0,1)$.

Since the pseudohyperbolic distance on $H^+$ is given by
$$\rho(w_1,w_2)=\left|\frac{w_1-w_2}{w_1-\overline{w_2}}\right|,$$
for any $z\in\mathbb{D}$, let $w=\sigma(z)$, then we have
\begin{align*}
d(z,\varphi_k(z))&=\rho(\sigma(z),\sigma\comp\varphi_k(z))
\\&=\rho(w,s^kw)
\\&=\left|\frac{w-s^kw}{w-s^k\overline{w}}\right|
\\&\geqslant\frac{1-s^k}{1+s^k}
\\&=\rho(i,s^ki)=d(z_0,\varphi_k(z_0)).
\end{align*}
Here $z_0=\sigma^{-1}(i)$. Since $\{\varphi_k(z_0)\}_{k=0}^{\infty}$ is an interpolating sequence,
$$\delta_{z_0}=\prod_{k=1}^{\infty}d(z_0,\varphi_k(z_0))>0.$$
So we have
$$\prod_{k=1}^{\infty}d(z,\varphi_k(z))\geqslant\delta_{z_0}>0$$
for all $z\in\mathbb{D}$.
\end{proof}

\begin{lemma}\label{int2}
Suppose $\varphi$ is a parabolic automorphism of $\mathbb{D}$ with Denjoy-Wolff point $a\in\partial\mathbb{D}$. Then for any $\epsilon>0$. there exist $\delta>0$ such that
$\prod_{k=1}^{\infty}d(z,\varphi_k(z))>\delta$ whenever $|z-a|>\epsilon$ and $z\in\mathbb{D}$.
\end{lemma}

\begin{proof}
Since $\varphi$ is a parabolic automorphism, we can find a biholomorphic map $\sigma$ from $\mathbb{D}$ to the upper half plane $H^+$ such that
$$\sigma\comp\varphi(z)=\sigma(z)+c$$
for all $z\in\mathbb{D}$. Here $c=\pm 1$.

Since $\sigma$ maps $a$ to $\infty$, we can find $R>0$ such that $|\sigma(z)|<R$ for all $z\in\mathbb{D}$ and $|z-a|>\epsilon$. Let $w=\sigma(z)$, then we have
\begin{align*}
d(z,\varphi_k(z))&=\rho(\sigma(z),\sigma\comp\varphi_k(z))
\\&=\rho(w,w+kc)
\\&=\left|\frac{w-w-kc}{w-\overline{w}-kc}\right|
\\&\geqslant\frac{k}{|2Ri-kc|}
\\&=\rho(Ri,Ri+kc)=d(z_0,\varphi_k(z_0)).
\end{align*}
Here $z_0=\sigma^{-1}(Ri)$. Since $\{\varphi_k(z_0)\}_{k=0}^{\infty}$ is an interpolating sequence,
$$\delta_{z_0}=\prod_{k=1}^{\infty}d(z_0,\varphi_k(z_0))>0.$$
So we have
$$\prod_{k=1}^{\infty}d(z,\varphi_k(z))\geqslant\delta_{z_0}>0$$
for all $z\in\mathbb{D}$ and $|z-a|>\epsilon$.
\end{proof}

\begin{lemma}\label{sp2}
Let $\varphi$ be a parabolic or hyperbolic automorphism of $\mathbb{D}$. Suppose $b\in\partial\mathbb{D}$ is the Denjoy-Wolff point of $\varphi^{-1}$ and $\psi\in H^\infty(\mathbb{D})$ is continuous at the point $b$. If $C_{\psi,\varphi}$ is not invertible, then the set
$$\{\lambda : |\lambda|\leqslant|\psi(b)|\varphi'(b)^{-1/2}\}$$
is contained in the spectrum of $C_{\psi,\varphi}$ on $H^2(\mathbb{D})$.
\end{lemma}

\begin{proof}
The case when $\psi(b)=0$ is trivial since $C_{\psi,\varphi}$ is not invertible. So we can assume that $\psi(b)\ne0$.

Since $C_{\psi,\varphi}$ is not invertible, we can find a sequence $\{z_j\}_{j=1}^{\infty}$ in $\mathbb{D}$ such that $\psi(z_j)$ tends to $0$ and $\psi(\varphi^{-1}_k(z_j))\ne 0$ for all $j\in\mathbb{N}$ and $k>0$.

Let $\{z_{j,k}\}_{k=0}^{\infty}$ be the iteration sequence of $z_j$ under $\varphi^{-1}$, that is, $z_{j,k}=\varphi^{-1}_k(z_j)$. Set $z_{j,-1}=\varphi(z_j)$. Then $z_{j,k}$ converges to $b$ as $k\to\infty$ for each $j\in\mathbb{N}$. Moreover, since $\psi(b)\ne0$, Remark \ref{r2} shows that the convergence is uniform with respect to $j$.

Let
$$e_{j,k}=\frac{K_{z_{j,k}}}{||K_{z_{j,k}}||}$$
denote the normalized reproducing kernel at the point $z_{j,k}$ in $H^2(\mathbb{D})$. Then we have
$$C_{\psi,\varphi}^*e_{j,k}=u_{j,k}e_{j,k-1},$$
where
$$u_{j,k}=\overline{\psi(z_{j,k})}\frac{||K_{z_{j,k-1}}||}{||K_{z_{j,k}}||}= \overline{\psi(z_{j,k})}\left(\frac{1-|z_{j,k}|^2}{1-|z_{j,k-1}|^2}\right)^{1/2}.$$
By Schwarz-Pick Theorem,
$$\frac{1-|z_{j,k-1}|^2}{1-|z_{j,k}|^2}=|\varphi'(z_{j,k})|,$$
so
$$\lim_{k\to\infty}u_{j,k}=\lim_{k\to\infty}\overline{\psi(z_{j,k})}|\varphi'(z_{j,k})|^{-1/2}=\overline{\psi(b)}\varphi'(b)^{-1/2}.$$
The limitation is uniform with respect to $j$.

Now fix $\lambda_0\in\{\lambda : |\lambda|<|\psi(b)|\varphi'(b)^{-1/2}\}$, and take $q>0$ such that $|\lambda_0|<q<|\psi(b)|\varphi'(b)^{-1/2}$. Then we can find $N>0$ such that
\begin{equation}\label{3.1}
\prod_{k=1}^{n-1}|u_{j,k}|>q^{n-1}
\end{equation}
for all $n>N$ and $j\in\mathbb{N}$.

Let $h_1=e_{1,0}$ and
$$h_n=e_{n,0}+\sum_{k=1}^{n-1}\lambda_0^ka_{n,k}e_{n,k}$$
when $n\geqslant 2$, where
$$a_{n,k}=\left(\prod_{s=1}^{k}u_{n,s}\right)^{-1}.$$
Then we have
$$(C_{\psi,\varphi}^*-\lambda_0)h_n=u_{n,0}e_{n,-1}-\frac{\lambda_0^n}{\prod_{k=1}^{n-1}u_{n,k}}e_{n,n-1}.$$
By (\ref{3.1}), when $n>N$,
$$||(C_{\psi,\varphi}^*-\lambda_0)h_n||\leqslant|u_{n,0}|+\frac{|\lambda_0|^n}{q^{n-1}}.$$
However, since $\psi(z_{n,0})=\psi(z_n)$ tends to $0$, we can know that
$$\lim_{n\to\infty}u_{n,0}=\lim_{n\to\infty}\overline{\psi(z_n)}|\varphi'(z_n)|^{-1/2}=0.$$
So $||(C_{\psi,\varphi}^*-\lambda_0)h_n||$ converges to zero as $n\to\infty$.

On the other hand, since  $\{z_{j,k}\}_{k=0}^{+\infty}$ is an interpolating sequence for each $j\in\mathbb{N}$, we can take $B_j(z)$ as the Blaschke product whose zero points are exactly $\{z_{j,k}\}_{k=1}^{\infty}$, then
$$|\langle B_ne_{n,0},h_n\rangle|=|\langle B_ne_{n,0},e_{n,0}\rangle|=\prod_{k=1}^{\infty}d(z_n,\varphi_{k}^{-1}(z_n)).$$
Since $\varphi^{-1}$ is also a parabolic or hyperbolic automorphism and $\{z_n\}$ converges to $z_0\ne b$, by Lemma \ref{int1} and Lemma \ref{int2} we can find $\delta>0$ such that
$$\prod_{k=1}^{\infty}d(z_n,\varphi_{k}^{-1}(z_n))>\delta$$
for all $n\in\mathbb{N}$, so
$$||h_n||\geqslant|\langle B_ne_{n,0},h_n\rangle|>\delta.$$ Thus we know that $\lambda_0$ belongs to the approximate spectrum of $C_{\psi,\varphi}^*$.
\end{proof}

Now we can get our final conclusions as follows.

\begin{theorem}\label{main2}
Let $\varphi$ be a hyperbolic automorphism of $\mathbb{D}$ with Denjoy-Wolff point $a$ and the other fixed point $b$. Assume that $\psi\in H^\infty(\mathbb{D})$ is continuous at the points $a$ and $b$. If $C_{\psi,\varphi}$ is not invertible, then the spectrum of $C_{\psi,\varphi}$ on $H^2(\mathbb{D})$ is
$$\left\{\lambda : |\lambda|\leqslant\max\{|\psi(a)|\varphi'(a)^{-1/2},|\psi(b)|\varphi'(b)^{-1/2}\}\right\}.$$
\end{theorem}

\begin{proof}
Since $\varphi$ is a hyperbolic automorphism, $\varphi^{-1}$ is also a hyperbolic automorphism and the Denjoy-Wolff point of $\varphi^{-1}$ is $b$. Then by Lemma \ref{sp2},
$$\{\lambda : |\lambda|\leqslant|\psi(b)|\varphi'(b)^{-1/2}\}\subset\sigma(C_{\psi,\varphi}).$$
So the conclusion follows directly from Lemma \ref{sp2} and Lemma \ref{rad3} when
$$|\psi(a)|\varphi'(a)^{-1/2}\leqslant|\psi(b)|\varphi'(b)^{-1/2}.$$
On the other hand, If
$$|\psi(a)|\varphi'(a)^{-1/2}>|\psi(b)|\varphi'(b)^{-1/2},$$
then the proof of Theorem 4.9 in \cite{Hyv} shows that
$$\{\lambda : |\psi(b)|\varphi'(b)^{-1/2}\leqslant|\lambda|\leqslant|\psi(a)|\varphi'(a)^{-1/2}\}$$
is contained in $\sigma(C_{\psi,\varphi})$. In fact, the proof there holds even if $\psi$ has zeros in $\overline{\mathbb{D}}$. Then along with Lemma \ref{sp2} and Lemma \ref{rad3}, we can also get our conclusion.
\end{proof}

\begin{theorem}\label{main3}
Let $\varphi$ be a parabolic automorphism of $\mathbb{D}$ with Denjoy-Wolff point $a\in\partial\mathbb{D}$. Assume that $\psi\in H^\infty(\mathbb{D})$ is continuous at the point $a$. If $C_{\psi,\varphi}$ is not invertible, then the spectrum of $C_{\psi,\varphi}$ on $H^2(\mathbb{D})$ is
$$\left\{\lambda : |\lambda|\leqslant|\psi(a)|\right\}.$$
\end{theorem}

\begin{proof}
Since $\varphi$ is a parabolic automorphism, $\varphi^{-1}$ is also a parabolic automorphism with Denjoy-Wolff point $a$. Then the result follows directly from Lemma \ref{rad2} and Lemma \ref{sp2}.
\end{proof}

\section{Weighted Bergman Spaces}

In this section, we give a brief illustration for the cases on weighted Bergman spaces. The situations are almost all the same, though some little differences arise.

\subsection{Elliptic Automorphisms}

When $\varphi$ is an elliptic automorphism, the methods used in Section $2$ are still available on weighted Bergman spaces. In fact, what really matter in our discussion are just the following requirements to the space.

\begin{req}
The space is a reproducing kernel Hilbert space.
\end{req}

\begin{req}
The reproducing kernel $K_z$ in the space has the following property: $||K_w||=||K_z||$ whenever $|z|=|w|$.
\end{req}

\begin{req}
For any $f\in H^\infty(\mathbb{D})$, $T_f$ is bounded on the space and $||T_f||=||f||_\infty$.
\end{req}

\begin{req}
Inequality (\ref{last2}) holds since we have $||z^n||=1$ in $H^2(\mathbb{D})$ for each $n\in\mathbb{N}$.
\end{req}

If the space being considered is a weighted Bergman space $A_\alpha^2(\mathbb{D})$, then Requirements $1$, $2$ and $3$ are satisfied obviously. For Requirement $4$, one only need to notice that $||z^n||_{A_\alpha^2}$ decreases with respect to $n\in\mathbb{N}$. Then for any $f\in A_\alpha^2(\mathbb{D})$ we have
$$||T_z^*f||_{A_\alpha^2}^2>||f||_{A_\alpha^2}^2-|f(0)|^2.$$
So (\ref{last2}) also holds on $A_\alpha^2(\mathbb{D})$.

Therefore, by carrying out a similar project as in Section $2$, we could get the following result.

\begin{theorem}\label{main4}
Let $\varphi$ be an elliptic automorphism of $\mathbb{D}$ with fixed point $a\in\mathbb{D}$ and $\psi\in H^\infty(\mathbb{D})$. Assume that $C_{\psi,\varphi}$ is not invertible.

If $\varphi'(a)^n=1$ for some integer $n\in\mathbb{N}$ and $n_0$ is the smallest such integer, then the spectrum of $C_{\psi,\varphi}$ on $A_\alpha^2(\mathbb{D})$ is the closure of the set
$$\{\lambda : \lambda^{n_0}=\prod_{k=0}^{n_0-1}\psi(\varphi_k(z)),z\in\mathbb{D}\}.$$

If $\varphi'(a)^n\ne 1$ for all $n\in\mathbb{N}$ and $\psi\in A(D)$, then the spectrum of $C_{\psi,\varphi}$ on $A_\alpha^2(\mathbb{D})$ is
$$\left\{\lambda:|\lambda|\leqslant|v(a)|\right\},$$
where $v$ is the outer part of $\psi$.
\end{theorem}

\subsection{Hyperbolic \& Parabolic Automorphisms}

The only requirement to the space in the proof of Lemma \ref{rad2} is that the spectral radius of composition operator $C_\varphi$ is $1$ on $H^2(\mathbb{D})$. This is also true on weighted Bergman spaces, see Lemma 7.2 in \cite{CM}.

In the proof of Lemma \ref{rad3}, we use the fact that $C_{(\varphi_n')^{1/2},\varphi_n}$ is an isometry on $H^2(\mathbb{D})$. When the space is $A_\alpha^2(\mathbb{D})$, we should take the operator $C_{(\varphi_n')^{\frac{\alpha+2}{2}},\varphi_n}$ instead. An easy calculation shows that this operator is an isometry on $A_\alpha^2(\mathbb{D})$, or one can just turn to the proof of Theorem 4.6 in \cite{Hyv}. Then a slight modification can show that the spectral radius of $C_{\psi,\varphi}$ on $A_\alpha^2(\mathbb{D})$ is
$$\max\left\{\frac{|\psi(a)|}{\varphi'(a)^{(\alpha+2)/2}},\frac{|\psi(b)|}{\varphi'(b)^{(\alpha+2)/2}}\right\}$$
when $\varphi$ is a hyperbolic automorphism.

For the proof of Lemma \ref{sp2}, one only need to note that the norm of $K_z$ is  $(1-|z|^2)^{-(\alpha+2)/2}$ in $A_\alpha^2(\mathbb{D})$. Then we can get our final results as follows.

\begin{theorem}\label{main5}
Let $\varphi$ be a hyperbolic automorphism of $\mathbb{D}$ with Denjoy-Wolff point $a\in\partial\mathbb{D}$ and the other fixed point $b\in\partial\mathbb{D}$. Assume that $\psi\in H^\infty(\mathbb{D})$ is continuous at the points $a$ and $b$. If $C_{\psi,\varphi}$ is not invertible on $A_\alpha^2(\mathbb{D})$, then the spectrum of $C_{\psi,\varphi}$ on $A_\alpha^2(\mathbb{D})$ is
$$\left\{\lambda :|\lambda|\leqslant\max\{\frac{|\psi(a)|}{\varphi'(a)^{(\alpha+2)/2}},\frac{|\psi(b)|}{\varphi'(b)^{(\alpha+2)/2}}\}\right\}.$$
\end{theorem}

\begin{theorem}\label{main6}
Let $\varphi$ be a parabolic automorphism of $\mathbb{D}$ with Denjoy-Wolff point $a\in\partial\mathbb{D}$ and $\psi\in H^\infty(\mathbb{D})$ is continuous at the point $a$. If $C_{\psi,\varphi}$ is not invertible on $A_\alpha^2(\mathbb{D})$, then the spectrum of $C_{\psi,\varphi}$ on $A_\alpha^2(\mathbb{D})$ is
$$\left\{\lambda:|\lambda|\leqslant|\psi(a)|\right\}.$$
\end{theorem}

\section{Further questions}

Theorem \ref{main11} and Theorem \ref{main4} are our main results for elliptic automorphisms. However, these result are got under the assumption that $\psi$ belongs to the disk algebra $A(\mathbb{D})$. We believe that the theorems still hold for the general cases when $\psi$ is not continuous up to the boundary of $D$, but it seems that some more work need to be done.

\begin{open1}
How to deal with the case when $\varphi$ is an elliptic automorphism and $\psi$ is not continuous up to the boundary of the unit disk $D$?
\end{open1}

When $\varphi$ is a parabolic or hyperbolic automorphism, we always assume that $\psi$ is continuous at the fixed points of $\varphi$. Though this assumption is reasonable, we still want to know the result for the more general cases. Particularly, the following question can be interesting.

\begin{open2}
What is the spectrum of $C_{\psi,\varphi}$ when $\varphi$ is a parabolic or hyperbolic automorphism and $\psi$ has no radial limitation at the fixed points of $\varphi$?
\end{open2}

\end{document}